\documentclass{amsart}

\usepackage[colorlinks=true, urlcolor=blue, citecolor=blue, linkcolor=blue, hypertexnames=false, hyperfootnotes=true]{hyperref}
\usepackage{amssymb}
\usepackage{mathrsfs}
\usepackage{dsfont}
\usepackage{mathtools}
\usepackage{tikz}

\newcommand\barrow{\textstyle\mathop{\rightarrow}_{}^{\hspace{-8pt}\bullet}}
\newcommand\arrowb{\textstyle\mathop{\rightarrow}_{\hspace{-8pt}\bullet}^{}}

\newcommand\carrow{\textstyle\mathop{\rightarrow}_{}^{\hspace{-8pt}\circ}}
\newcommand\arrowc{\textstyle\mathop{\rightarrow}_{\hspace{-8pt}\circ}^{}}
\newcommand\carrowc{\textstyle\mathop{\rightarrow}_{\hspace{-8pt}\circ}^{\hspace{-8pt}\circ}}
\newcommand\carrowb{\textstyle\mathop{\rightarrow}_{\hspace{-8pt}\bullet}^{\hspace{-8pt}\circ}}

\newtheorem{thm}{Theorem}

\newtheorem{prp}{Proposition}

\newtheorem{cor}{Corollary}

\theoremstyle{definition}

\newtheorem{dfn}{Definition}

\newtheorem{qst}{Question}

\numberwithin{equation}{section}

\author{Tristan Bice}
\address{Federal University of Bahia\\
Salvador\\
Brazil}
\email{Tristan.Bice@gmail.com}
\thanks{This research has been supported by an IMPA (Brazil) postdoctoral fellowship.}
\keywords{distance, hemimetric, quasimetric, order, topology, complete}
\subjclass[2010]{06A06, 18A35, 54E50, 54E55}

\title{Yoneda Completeness}

\begin{document}

\begin{abstract}
We characterize Yoneda completeness for non-symmetric distances by combinations of metric and directed completeness.  One of these generalizes the Kostanek-Waszkiewicz theorem on formal balls.
\end{abstract}

\maketitle

\section*{Motivation}

Yoneda completeness was introduced in \cite{Wagner1997} and \cite{Bonsangue1998} to unify metric and order theoretic notions of completeness.  More precisely, the goal was to find a natural notion of completeness for non-symmetric distances that reduces to Cauchy completeness in the metric case and directed completeness in the partial order case.  We aim to take this further by showing that, even in more general distance spaces, Yoneda completeness can still be characterized by several different combinations of metric and directed completeness.

We draw our inspiration from a perhaps surprising source, namely C*-algebra semicontinuity theory (see \cite{AkemannPedersen1973} and \cite{Brown1988}).  Various order relations in C*-algebras can be composed with the metric to form non-symmetric distances, although they are never mentioned explicitly in the C*-algebra literature.  This is unfortunate, as non-symmetric distances could simplify and generalize certain aspects of C*-algebra theory.  In particular this rings true for C*-algebra semicontinuity theory, where some sophisticated C*-algebraic machinery can be replaced by the elementary net manipulations that we describe here.  This will also no doubt have applications to distance spaces that arise in other areas of algebra and analysis.

\section*{Outline}

In \autoref{P} we give the basic the definitions and theory of (pre-)Cauchy nets, ball and hole topologies, non-symmetric distances and supremums.  We take \cite{Wagner1997} and \cite{Bonsangue1998} as our primary references although our approach is slightly more general, e.g. we deal with distances rather than hemimetrics and nets rather than sequences.  Although to keep things simple, the range of our distance functions will always be the positive extended real line $[0,\infty]$ as in \cite{Bonsangue1998}, rather than the more general quantales considered in \cite{Wagner1997}.  For the completeness notions we consider, see \autoref{YC} and \autoref{edcomplete} respectively.

In \autoref{Completeness}, we construct several closely related sequences and subsets from a given Cauchy net $(x_\lambda)$.  Their consequences regarding completeness are collected at the end in \autoref{Yc}.  We finish with a simple application to formal balls in \autoref{KW}, showing that \autoref{Yc} \eqref{Yc1a} generalizes the Kostanek-Waszkiewicz theorem.

\section{Preliminaries}\label{P}

We make the following standing assumption.
\[\textbf{$\mathbf{d}$ and $\mathbf{e}$ are functions from $X\times X$ to $[0,\infty]$.}\]

\subsection{Nets}
The nets $(x_\lambda)\subseteq X$ we will be concerned with are defined as follows.
\begin{align}
\lim_\gamma\liminf_\delta\mathbf{d}(x_\gamma,x_\delta)=0\quad&\Leftrightarrow\quad(x_\lambda)\text{ is \emph{$\mathbf{d}$-reflexive}}.\\
\lim_\gamma\limsup_\delta\mathbf{d}(x_\gamma,x_\delta)=0\quad&\Leftrightarrow\quad(x_\lambda)\text{ is \emph{$\mathbf{d}$-pre-Cauchy}}.\\
\label{Cauchy}\lim_{\gamma\prec\delta}\mathbf{d}(x_\gamma,x_\delta)=0\quad&\Leftrightarrow\quad(x_\lambda)\text{ is \emph{$\mathbf{d}$-Cauchy}}.
\end{align}
Just to be clear, by a net we mean a set indexed by a directed set $\Lambda$, i.e. there is a (possibly non-reflexive) transitive relation $\mathbin{\prec}\subseteq\Lambda\times\Lambda$ satisfying $\forall\gamma,\delta\ \exists\lambda\ (\gamma,\delta\prec\lambda)$, with limits inferior and superior defined by
\begin{align*}
\liminf_\lambda r_\lambda&=\lim_\gamma\inf_{\gamma\prec\lambda}r_\lambda.\\
\limsup_\lambda r_\lambda&=\lim_\gamma\sup_{\gamma\prec\lambda}r_\lambda.
\end{align*}
And in \eqref{Cauchy} we consider $\prec$ itself as a directed subset of $\Lambda\times\Lambda$ ordered by $\prec\times\prec$.

The above nets can also be characterized by a filter $\Phi^\mathbf{d}\subseteq\mathscr{P}(X\times X)$ defined from $\mathbf{d}$.  Specifically, for $\mathbin{\prec}\subseteq[0,\infty]\times[0,\infty]$ and $\epsilon\in[0,\infty]$, define $\mathbin{\prec^\mathbf{d}_\epsilon}\subseteq X\times X$ by
\[x\prec^\mathbf{d}_\epsilon y\quad\Leftrightarrow\quad\mathbf{d}(x,y)\prec\epsilon.\]
Taking the usual $<$ on $[0,\infty]$ for $\prec$, we define
\[\Phi^\mathbf{d}=\{\mathbin{\preceq}:\epsilon>0\text{ and }\mathbin{<^\mathbf{d}_\epsilon}\,\subseteq\,\mathbin{\preceq}\,\subseteq\, X\times X\}.\]
So $\mathbin{\leq^\mathbf{d}}=\mathbin{\leq^\mathbf{d}_0}=\bigcap\Phi^\mathbf{d}$ and
\begin{align*}
\forall\mathbin{\preceq}\in\Phi^\mathbf{d}\ \exists\alpha\ \forall\gamma\succ\alpha\ \forall\beta\ \exists\delta\succ\beta\ (x_\gamma\preceq x_\delta)\quad&\Leftrightarrow\quad(x_\lambda)\text{ is \emph{$\mathbf{d}$-reflexive}}.\\
\forall\mathbin{\preceq}\in\Phi^\mathbf{d}\ \exists\alpha\ \forall\gamma\succ\alpha\ \exists\beta\ \forall\delta\succ\beta\ (x_\gamma\preceq x_\delta)\quad&\Leftrightarrow\quad(x_\lambda)\text{ is \emph{$\mathbf{d}$-pre-Cauchy}}.\\
\forall\mathbin{\preceq}\in\Phi^\mathbf{d}\ \exists\alpha\ \forall\gamma\succ\alpha\ \hspace{16pt}\forall\delta\succ\gamma\ (x_\gamma\preceq x_\delta)\quad&\Leftrightarrow\quad(x_\lambda)\text{ is \emph{$\mathbf{d}$-Cauchy}}.
\end{align*}

We immediately see that
\begin{center}
$\mathbf{d}$-Cauchy$\quad\Rightarrow\quad\mathbf{d}$-pre-Cauchy$\quad\Rightarrow\quad\mathbf{d}$-reflexive.
\end{center}
Denote the finite subsets of $\Lambda$ by $[\Lambda]^{<\omega}$, i.e. with $|F|$ denoting $F$'s cardinality,
\[[\Lambda]^{<\omega}\ =\ \{F\subseteq\Lambda:|F|<\omega\}.\]

\begin{prp}\label{preCauchysub}
Any $\mathbf{d}$-pre-Cauchy net $(x_\lambda)\subseteq X$ has a $\mathbf{d}$-Cauchy subnet.
\end{prp}

\begin{proof}
If $\Lambda$ is finite then it has a maximum $\gamma$, which means the single element net $x_\gamma$ is a $\mathbf{d}$-Cauchy subnet.  Otherwise, define a map $f:[\Lambda]^{<\omega}\rightarrow\Lambda$ recursively as follows.  Let $f(\{\lambda\})=\lambda$, for all $\lambda\in\Lambda$.  As $(x_\lambda)$ is $\mathbf{d}$-pre-Cauchy, given any other $F\in[\Lambda]^{<\omega}$ we can take $f(F)\in\Lambda$ such that, for all $E\subsetneqq F$, $f(E)\prec f(F)$ and
\[\mathbf{d}(x_{f(E)},x_{f(F)})\leq\underset{\lambda}{\lim\sup}\,\mathbf{d}(x_{f(E)},x_\lambda)+2^{-|F|}.\]
Then $(x_{f(F)})$ is a $\mathbf{d}$-Cauchy subnet of $(x_\lambda)$ w.r.t. $\subsetneqq$ on $[X]^{<\omega}$.
\end{proof}

When $\mathbf{d}$ is a metric, there is usually only one type of net of interest as
\begin{align*}
\text{$\mathbf{d}$-Cauchy}\quad&\Leftrightarrow\quad\text{$\mathbf{d}$-pre-Cauchy}.\\
\text{$\mathbf{d}$-reflexive}\quad&\Leftrightarrow\quad\text{arbitrary},\quad\text{if $X$ is totally bounded}.
\end{align*}
On the other hand, for any partial order $\mathbin{\preceq}\subseteq X\times X$,
\[\text{$\preceq$-Cauchy}\quad\Leftrightarrow\quad\text{eventually increasing},\]
when we identify $\preceq$ with its characteristic function (as we do from now on)
\[\preceq(x,y)=\begin{cases}0&\text{if }x\preceq y\\ \infty&\text{otherwise}\end{cases}\]
(e.g. $\leq^\mathbf{d}$ is identified with $\infty\mathbf{d}$, taking $\infty0=0$).
In this case there are simple examples of non-Cauchy pre-Cauchy sequences \textendash\, see \cite{Wagner1997} Remark 2.11.

\subsection{Topology}  For any $\mathbin{\prec}\subseteq X\times X$, we define
\begin{align*}
x\prec\quad&=\quad\{y\in X:x\prec y\}.\\
\prec x\quad&=\quad\{y\in X:y\prec x\}.
\end{align*}
Define the open upper/lower balls/holes with centre $c\in X$ and radius $\epsilon$ by
\begin{align*}
c^\bullet_\epsilon\quad=\quad c\mathrel{<^\mathbf{d}_\epsilon}\quad&=\quad\{x\in X:\mathbf{d}(c,x)<\epsilon\}.\\
c_\bullet^\epsilon\quad=\hspace{13pt}\mathrel{<^\mathbf{d}_\epsilon}c\quad&=\quad\{x\in X:\mathbf{d}(x,c)<\epsilon\}.\\
c^\circ_\epsilon\quad=\hspace{13pt}\mathrel{>^\mathbf{d}_\epsilon}c\quad&=\quad\{x\in X:\mathbf{d}(x,c)>\epsilon\}.\\
c_\circ^\epsilon\quad=\quad c\mathrel{>^\mathbf{d}_\epsilon}\quad&=\quad\{x\in X:\mathbf{d}(c,x)>\epsilon\}.
\end{align*}
Let $X^\bullet$, $X_\bullet$, $X^\circ$, $X_\circ$, $X^\bullet_\bullet$, $X^\bullet_\circ$, $X^\circ_\bullet$ and $X^\circ_\circ$ denote the topologies generated by the corresponding balls and holes, i.e. by arbitrary unions of finite intersections.  Denote convergence by $\barrow$, $\arrowb$, $\carrow$, $\arrowc$, etc. so, for any net $(x_\lambda)\subseteq X$,
\begin{align*}
x_\lambda\barrow x\quad&\Leftrightarrow\quad\forall c\in X\ \lim\sup\mathbf{d}(c,x_\lambda)\leq\mathbf{d}(c,x).\\
x_\lambda\arrowb x\quad&\Leftrightarrow\quad\forall c\in X\ \lim\sup\mathbf{d}(x_\lambda,c)\leq\mathbf{d}(x,c).\\
x_\lambda\carrow x\quad&\Leftrightarrow\quad\forall c\in X\ \lim\inf\ \mathbf{d}(x_\lambda,c)\geq\mathbf{d}(x,c).\\
x_\lambda\arrowc x\quad&\Leftrightarrow\quad\forall c\in X\ \lim\inf\ \mathbf{d}(c,x_\lambda)\geq\mathbf{d}(c,x).
\end{align*}

Most of the literature on non-symmetric distances has focused on ball topologies (one of the few places hole topologies are mentioned is \cite{Goubault2013} Exercise 6.2.11).  However, it is really the hole topologies that are more intimately connected to the $\leq^\mathbf{d}$ order structure.  The double hole topology also defines our central concept.
\begin{dfn}\label{YC}
$X$ is \emph{$\mathbf{d}$-complete} if every $\mathbf{d}$-Cauchy net has a $X^\circ_\circ$-limit.
\end{dfn}
This was called $\liminf$-completeness in \cite{Wagner1997} and just completeness in \cite{Bonsangue1998}, although the original formulations differ somewhat from \autoref{YC} \textendash\, see the comments after \autoref{preCauchyconvergence}.  These days it is usually called Yoneda completeness to distinguish it from other similar notions (e.g. Smyth completeness where $X^\bullet_\bullet$ is considered instead of $X^\circ_\circ$ \textendash\, see \cite{Smyth1988}) but these will not be discussed here.

Let us point out that, while $\mathbf{d}$-Cauchy nets depend only $\Phi^\mathbf{d}$, the double hole topology $X^\circ_\circ$ depends crucially on $\mathbf{d}$, i.e. $\mathbf{d}$-completeness is not a `uniform property'.  Below we will use uniform concepts where possible, but the inherent non-uniform nature of $\mathbf{d}$-completeness means there is a limit to how much this can be done.

\subsection{Distances}  For $x\in X$, define $x\mathbf{d},\mathbf{d}x:X\rightarrow[0,\infty]$ by
\begin{align*}
x\mathbf{d}(y)&=\mathbf{d}(x,y).\\
\mathbf{d}x(y)&=\mathbf{d}(y,x).
\end{align*}
The composition of $\mathbf{d}$ and $\mathbf{e}$ is defined by
\[\mathbf{d}\circ\mathbf{e}(x,y)=\inf_{z\in X}(x\mathbf{d}+\mathbf{e}y)(z).\]
We call $\mathbf{d}$ a \emph{distance} if
\[\mathbf{d}\leq\mathbf{d}\circ\mathbf{d}.\tag{$\triangle$}\label{tri}\]
This implies $\mathbin{\leq^\mathbf{d}}\circ\mathbin{\leq^\mathbf{d}}\,\subseteq\,\mathbin{\leq^\mathbf{d}}$, i.e. $\mathbin{\leq^\mathbf{d}}$ is transitive.  As in \cite{Goubault2013} Definition 6.1.1, we call $\mathbf{d}$ a \emph{hemimetric} if $\leq^\mathbf{d}$ is also reflexive, i.e. a preorder.

Non-hemimetric distances have rarely been considered until now.  Requiring $\leq^\mathbf{d}$ to be reflexive may seem harmless, but there are indeed natural distances for which this fails, e.g. $\mathbf{d}(x,y)=x(1-y)$ on $[0,1]$ or its extension to the positive unit ball of an arbitrary C*-algebra.  There are also natural constructions which preserve \eqref{tri} but not $\leq^\mathbf{d}$-reflexivity.  For example, just as one extends $\mathbf{d}$ to a distance on subsets of $X$ in the Hausdorff-Hoare construction (see \cite{Goubault2013} Lemma 7.5.1), one can extend $\mathbf{d}$ to a distance on nets in $X$ as in \cite{Wagner1997} Proposition 2.6 by
\begin{equation}\label{HH}
\mathbf{d}((x_\lambda),(y_\gamma))=\limsup_\lambda\liminf_\gamma\mathbf{d}(x_\lambda,y_\gamma).
\end{equation}
However, even if $\leq^\mathbf{d}$ is reflexive on $X$, $\leq^\mathbf{d}$ may not be reflexive on all nets.  Indeed
\[(x_\lambda)\text{ is $\mathbf{d}$-reflexive}\quad\Leftrightarrow\quad(x_\lambda)\leq^\mathbf{d}(x_\lambda).\]
Moreover, the extra generality comes at little cost.  So let us now on assume that
\[\mathbf{d}\textbf{ and $\mathbf{e}$ are arbitrary distances on }X.\]

Now hole limits of $\mathbf{d}$-reflexive $(x_\lambda)\subseteq X$ can be characterized as follows.
\begin{align}
\label{arrowc}x_\lambda\arrowc x\quad&\Leftrightarrow\quad\mathbf{d}(x_\lambda,x)\rightarrow0.\\
\label{carrowc}x_\lambda\carrowc x\quad&\Leftrightarrow\quad x_\lambda\carrowb x\leq^\mathbf{d}x.
\end{align}

\begin{proof}\
\begin{itemize}
\item[\eqref{arrowc}]  If $x_\lambda\arrowc x$ then $\lim_\gamma\mathbf{d}(x_\gamma,x)\leq\lim_\gamma\liminf_\lambda\mathbf{d}(x_\gamma,x_\lambda)=0$.  If $\mathbf{d}(x_\lambda,x)\rightarrow0$ then $\mathbf{d}(c,x)\leq\lim\inf\mathbf{d}(c,x_\lambda)+\mathbf{d}(x_\lambda,x)=\lim\inf\mathbf{d}(c,x_\lambda)$, for any $c\in X$.

\item[\eqref{carrowc}]  If $\mathbf{d}(x_\lambda,x)\rightarrow0$ then $\mathbf{d}(x_\lambda,c)\leq\mathbf{d}(x_\lambda,x)+\mathbf{d}(x,c)\rightarrow\mathbf{d}(x,c)$ so $x_\lambda\arrowb x$.  If $x_\lambda\carrow x$ too then $\mathbf{d}(x,x)\leq\liminf\mathbf{d}(x_\lambda,x)=0$, i.e. $x\leq^\mathbf{d}x$.  Conversely, if $x_\lambda\arrowb x\leq^\mathbf{d}x$ then $\limsup\mathbf{d}(x_\lambda,x)\leq\mathbf{d}(x,x)=0$, i.e. $\mathbf{d}(x_\lambda,x)\rightarrow0$.\qedhere
\end{itemize}
\end{proof}

For an example of $\mathbf{d}$-reflexive $x_\lambda\carrowb x\nleq^\mathbf{d}x$, take any $x_\lambda\rightarrow0<x$ in $[0,\infty)$ where, for the distance $\mathbf{d}$, we simply consider the coordinate projection $\mathbf{d}(y,z)=z$.

In \cite{Goubault2013} Definition 7.1.15, any $x$ which satisfies $\mathbf{d}(x,y)=\limsup\mathbf{d}(x_\lambda,y)$, for all $y\in X$, is called a \emph{$\mathbf{d}$-limit} of $(x_\lambda)$ (these are called \emph{forward limits} in \cite{Bonsangue1998} before Proposition 3.3 and just \emph{limits} in \cite{KunziSchellekens2002} Definition 11).  In general, $\mathbf{d}$-limits are not true limits in any topological sense, as they are not preserved by taking subnets.  But for $\mathbf{d}$-pre-Cauchy nets, $\mathbf{d}$-limits are $X^\circ_\bullet$-limits, i.e. the limit superior will be a limit, as shown below and in \cite{Wagner1997} Theorem 2.26.

\begin{prp}\label{Clim}
For $\mathbf{d}$-pre-Cauchy $(x_\lambda)$ and $y\in X$, $\mathbf{d}(x_\lambda,y)$ and $\mathbf{d}(y,x_\lambda)$ converge.
\end{prp}

\begin{proof}  As $(x_\lambda)$ is $\mathbf{d}$-pre-Cauchy,
\begin{align*}
\limsup\limits_\lambda\mathbf{d}(x_\lambda,y)&\leq\limsup\limits_\lambda\liminf\limits_\gamma\mathbf{d}(x_\lambda,x_\gamma)+\mathbf{d}(x_\gamma,y)=\liminf\limits_\gamma\mathbf{d}(x_\gamma,y).\\
\liminf_\lambda\mathbf{d}(y,x_\lambda)&\geq\liminf_\lambda\limsup_\gamma\mathbf{d}(y,x_\gamma)-\mathbf{d}(x_\lambda,x_\gamma)=\limsup_\gamma\mathbf{d}(y,x_\gamma).\qedhere
\end{align*}
\end{proof}

\begin{cor}\label{preCauchyconvergence}
Any $\mathbf{d}$-pre-Cauchy net converges in $X^\bullet$, $X^\circ$, $X_\bullet$ or $X_\circ$ iff it has a subnet that converges in the same topology.
\end{cor}

For distance $\mathbf{d}$, we could thus replace $\mathbf{d}$-Cauchy nets with $\mathbf{d}$-pre-Cauchy nets in \autoref{YC}, by \autoref{preCauchysub} and \autoref{preCauchyconvergence}.  And for hemimetric $\mathbf{d}$, \autoref{YC} agrees with the $\mathbf{d}$-limit definition of Yoneda completeness in \cite{Goubault2013} Definition 7.4.1.  We prefer $X^\circ_\circ$-limits to $X^\circ_\bullet$-limits/$\mathbf{d}$-limits for the following reasons.
\begin{enumerate}
\item $X^\circ_\circ$ seems more natural for general distances (e.g. $\mathbf{d}(x,y)=y$ noted above).
\item $X^\circ_\circ$ is self-dual, making it clear that the asymmetry in $\mathbf{d}$-completeness comes from the nets being considered rather than the topology.
\item $X^\circ_\circ$ already arises naturally in various situations (although this does not appear to be widely recognized), e.g. as the usual product topology for products of bounded intervals, as the Wijsman topology for subsets of $X$, and as the weak operator topology for projections on a Hilbert space.
\end{enumerate}

If $\mathbf{d}$ is a metric then limits of $\mathbf{d}$-Cauchy nets are the same in $X^\circ_\circ$ and $X^\bullet=X_\bullet$.  Thus $\mathbf{d}$-completeness generalizes the usual notion of metric completeness.  If we consider (the characteristic function of) a partial order $\mathbin{\preceq}\subseteq X\times X$ then $X^\circ_\circ$-limits of increasing nets are precisely their supremums, so $\mathbf{d}$-completeness also generalizes directed completeness.  Our main thesis is that, even in more general distances spaces, $\mathbf{d}$-completeness is a combination of metric and directed completeness.  To make this precise we need to extend the usual order theoretic notion of supremum.

\subsection{Supremums}  For $Y\subseteq X$ define
\begin{align*}
Y\mathbf{d}&=\sup y\mathbf{d}.\\
\mathbf{d}Y&=\inf\mathbf{d}y.
\end{align*}
Also define $Y\leq^\mathbf{d}x\ \Leftrightarrow\ Y\subseteq(\leq^\mathbf{d}x)$.  We define \emph{$\mathbf{d}$-supremums} of $Y\subseteq X$ by
\[x=\text{$\mathbf{d}$-$\sup Y$}\qquad\Leftrightarrow\qquad x\mathbf{d}=Y\mathbf{d}\quad\text{and}\quad Y\leq^\mathbf{d}x.\]
Note $=$ is a slight abuse of notation, as $\mathbf{d}$-supremums are only unique up to the equivalence relation $x\leq^\mathbf{d}y\leq^\mathbf{d}x$.  Also, we could replace $x\mathbf{d}=Y\mathbf{d}$ with $x\mathbf{d}\leq Y\mathbf{d}$, as $Y\leq^\mathbf{d}x\Rightarrow Y\mathbf{d}\leq x\mathbf{d}$.  Alternatively, we could replace $Y\leq^\mathbf{d}x$  with $x\leq^\mathbf{d}x$ as $x\mathbf{d}=Y\mathbf{d}$ implies $x\leq^\mathbf{d}x\,\Leftrightarrow\, x\mathbf{d}(x)=0\,\Leftrightarrow\, Y\mathbf{d}(x)=0\,\Leftrightarrow\, Y\leq^\mathbf{d}x$.

Note $\mathbf{d}$-supremums are $\leq^\mathbf{d}$-supremums, as $x\mathbf{d}=Y\mathbf{d}$ implies $\infty x\mathbf{d}=\infty Y\mathbf{d}$.  However, unless we place some extra condition on $\mathbf{d}$, the converse can fail e.g. if $\mathbf{d}(r,s)=(r-s)_+$ (where $r_+=r\vee0$) on $X=[0,1)\cup\{2\}$ then we see that \linebreak $2=\mathbin{\leq^\mathbf{d}}$-$\sup[0,1)\neq\mathbf{d}$-$\sup[0,1)$, as $\sup_{x\in[0,1)}\mathbf{d}(x,0)=1\neq2=\mathbf{d}(2,0)$.  Indeed, in general $\mathbf{d}$-supremums depend crucially on $\mathbf{d}$, not just $\leq^\mathbf{d}$ or even $\Phi^\mathbf{d}$.

One such condition would be `every closed lower ball has a maximum'.  In fact, something weaker suffices.  Specifically, consider the following functions on $[0,\infty]$.
\begin{align*}
\mathbf{d}^\bullet(r)&=\sup_{x\in X}\inf_{y\leq^\mathbf{d}x^\bullet_r}\mathbf{d}(x,y).\\
\mathbf{d}_\bullet(r)&=\sup_{x\in X}\inf_{x_\bullet^r\leq^\mathbf{d}y}\mathbf{d}(y,x).
\end{align*}
Also let $\mathbf{I}$ denote the identity on $[0,\infty]$ so
\[\mathbf{d}_\bullet\leq\mathbf{I}\quad\Leftrightarrow\quad\sup_{y\in Y}\mathbf{d}(y,x)=\inf_{Y\leq^\mathbf{d}y}\mathbf{d}(y,x)\text{ whenever }x\in X\supseteq Y.\]

\begin{prp}\label{Xcompdirected}
If $\mathbf{d}_\bullet\leq\mathbf{I}$ then $\leq^\mathbf{d}$-supremums are $\mathbf{d}$-supremums.
\end{prp}

\begin{proof}
Assume $Y\subseteq X$ and $z=\mathbin{\leq^\mathbf{d}}$-$\sup Y\neq\mathbf{d}$-$\sup Y$ so $\sup_{y\in Y}\mathbf{d}(y,x)<\mathbf{d}(z,x)$, for some $x\in X$.  As $\mathbf{d}_\bullet\leq\mathbf{I}$, we have $w\in X$ with $Y\leq^\mathbf{d}w$ and $\mathbf{d}(w,x)<\mathbf{d}(z,x)$.  But then $z=\mathbin{\leq^\mathbf{d}}$-$\sup Y\leq^\mathbf{d}w$ so $\mathbf{d}(z,x)\leq\mathbf{d}(w,x)$, a contradiction.
\end{proof}

We also need to generalize directedness.  Specifically, for $Y\subseteq X$ we define
\[Y\text{ is \emph{$\mathbf{d}$-directed}}\quad\Leftrightarrow\quad\forall F\in[Y]^{<\omega}\inf_{y\in Y}F\mathbf{d}(y)=0.\]
By \eqref{tri}, $[Y]^{<3}$ suffices.  Also define $Y\leq^\mathbf{d}(x_\lambda)\,\Leftrightarrow\,\mathbf{d}(y,x_\lambda)\rightarrow0$, for all $y\in Y$, so
\begin{align*}
Y\leq^\mathbf{d}(x_\lambda)\subseteq Y\quad&\Rightarrow\quad (x_\lambda)\text{ is $\mathbf{d}$-pre-Cauchy}.\\
\exists(x_\lambda)\ Y\leq^\mathbf{d}(x_\lambda)\subseteq Y\quad&\Leftrightarrow\quad Y\text{ is $\mathbf{d}$-directed}.
\end{align*}
Indeed, if $Y$ is $\mathbf{d}$-directed then, for $F\in[Y]^{<\omega}$ and $\epsilon>0$, take $y_{F,\epsilon}\in Y$ with $F\mathbf{d}(y_{F,\epsilon})<\epsilon$, so $Y\leq^\mathbf{d}(y_{F,\epsilon})\subseteq Y$, ordering $[Y]^{<\omega}\times(0,\infty)$ by $\subseteq\times\geq$.

\begin{dfn}\label{edcomplete}
$X$ is \emph{$\mathbf{e}$-$\mathbf{d}$-complete} if every $\mathbf{e}$-directed $Y\subseteq X$ has a $\mathbf{d}$-supremum.
\end{dfn}

If $\mathbin{\preceq}\subseteq X\times X$ is a partial order, $\preceq$-$\preceq$-completeness is directed completeness.  Thus both $\mathbf{d}$-$\mathbf{d}$-completeness and $\leq^\mathbf{d}$-$\mathbf{d}$-completeness are valid generalizations.  But if $\mathbf{d}$ is a metric then every $\mathbf{d}$-directed subset contains at most $1$ element, which makes $X$ trivially $\mathbf{d}$-$\mathbf{d}$-complete.  So, unlike $\mathbf{d}$-completeness, $\mathbf{d}$-$\mathbf{d}$-completeness does not generalize metric completeness.  In general, $\mathbf{d}$-completeness is a stronger notion, as we now show.

\begin{prp}\label{a1lim}
If $Y\leq^\mathbf{d}(x_\lambda)\subseteq Y$ and $x\in X$ then
\begin{align}
x_\lambda\arrowc x\quad&\Leftrightarrow\quad Y\leq^\mathbf{d}x.\label{a1holelim}\\
x_\lambda\carrow x\quad&\Leftrightarrow\quad x\mathbf{d}\leq Y\mathbf{d}.\label{a1limhole}\\
x_\lambda\carrowc x\quad &\Leftrightarrow\quad x=\mathbf{d}\text{-}\sup Y.\label{a1doubleholelim}
\end{align}
\end{prp}

\begin{proof}
\item{\eqref{a1holelim}} If $Y\leq^\mathbf{d}x$ then $\mathbf{d}(x_\lambda,x)=0$, for all $\lambda$, so $x_\lambda\arrowc x$, by \eqref{arrowc}.  While if $x_\lambda\arrowc x$ and $y\in Y$ then $\mathbf{d}(y,x)\leq\liminf\mathbf{d}(y,x_\lambda)=0$, as $Y\leq^\mathbf{d}(x_\lambda)$, i.e. $Y\leq^\mathbf{d}x$. 

\item{\eqref{a1limhole}}  If $x\mathbf{d}\leq Y\mathbf{d}$ then, as $Y\leq^\mathbf{d}(x_\lambda)$, for any $z\in X$ we have
\[\mathbf{d}(x,z)\leq\sup_{y\in Y}\mathbf{d}(y,z)\leq\sup_{y\in Y}\liminf(\mathbf{d}(y,x_\lambda)+\mathbf{d}(x_\lambda,z))=\liminf\mathbf{d}(x_\lambda,z).\]
While if $x_\lambda\carrow x$ then $x\mathbf{d}(z)=\mathbf{d}(x,z)\leq\liminf\mathbf{d}(x_\lambda,z)\leq Y\mathbf{d}(z)$, for all $z\in X$.

\item{\eqref{a1doubleholelim}} See \eqref{a1holelim} and \eqref{a1limhole}.
\end{proof}

\begin{cor}
If $X$ is $\mathbf{d}$-complete then $X$ is $\mathbf{d}$-$\mathbf{d}$-complete.
\end{cor}

\begin{proof}
For any $\mathbf{d}$-directed $Y\subseteq X$, take $(x_\lambda)$ with $Y\leq^\mathbf{d}(x_\lambda)\subseteq Y$.  By \autoref{preCauchysub}, we may revert to a Cauchy subnet (which still satisfies $Y\leq^\mathbf{d}(x_\lambda)$).  As $X$ is $\mathbf{d}$-complete, $x_\lambda\carrowc x$, for some $x\in X$.  By \eqref{a1doubleholelim}, $x$ is a $\mathbf{d}$-supremum of $Y$.
\end{proof}

For $\mathbf{d}$-pre-Cauchy $(x_\lambda)\subseteq X$, it will also be convenient to define
\begin{align*}
(x_\lambda)\mathbf{d}&=\lim x_\lambda\mathbf{d}.\\
\mathbf{d}(x_\lambda)&=\lim\mathbf{d}x_\lambda.
\end{align*}
It then follows immediately from the definitions that
\begin{align*}
x_\lambda\arrowc x\quad&\Leftrightarrow\quad\mathbf{d}x\leq\mathbf{d}(x_\lambda).\\
x_\lambda\carrow x\quad&\Leftrightarrow\quad x\mathbf{d}\leq(x_\lambda)\mathbf{d}.\\
x_\lambda\carrowb x\quad&\Leftrightarrow\quad x\mathbf{d}=(x_\lambda)\mathbf{d}.
\end{align*}

\begin{prp}
If $Y$ is $\mathbf{d}$-directed and $(x_\lambda)\subseteq X$ is $\mathbf{d}$-pre-Cauchy then
\[Y\leq^\mathbf{d}(x_\lambda)\quad\Leftrightarrow\quad\mathbf{d}(x_\lambda)\leq\mathbf{d}Y.\]
\end{prp}

\begin{proof}
If $y\in Y\leq^\mathbf{d}(x_\lambda)$ and $x\in X$ then $\mathbf{d}(x,x_\lambda)\leq\mathbf{d}(x,y)+\mathbf{d}(y,x_\lambda)\rightarrow\mathbf{d}(x,y)$, i.e. $\mathbf{d}(x_\lambda)\leq\mathbf{d}y$, for all $y\in Y$, so $\mathbf{d}(x_\lambda)\leq\mathbf{d}Y$.  While if $\mathbf{d}(x_\lambda)\leq\mathbf{d}Y$ and $y\in Y$ then $\lim\mathbf{d}(y,x_\lambda)\leq\mathbf{d}Y(y)=0$, as $Y$ is $\mathbf{d}$-directed, i.e. $Y\leq^\mathbf{d}(x_\lambda)$.
\end{proof}

\section{Cauchy Nets}\label{Completeness}

In this section we make the following standing assumption
\[(x_\lambda)\subseteq X\textbf{ is $\mathbf{d}$-Cauchy}.\]

For our first result we could assume `every closed upper ball has a minimum'.  As in \autoref{Xcompdirected}, we can weaken this to $\mathbf{d}^\bullet\leq\mathbf{I}$, but here even $\mathbf{d}^\bullet\precapprox\mathbf{I}$ suffices, where $\precapprox$ is `uniform subequivalence'.  Specifically, for $f,g:X\rightarrow[0,\infty]$, define
\begin{align*}
\sup_{g(x)\leq r}f(x)\ &=\ ^f\!/\!_g(r)\\
f\precapprox g\ &\Leftrightarrow\ ^f\!/\!_g(r)\rightarrow0.
\end{align*}
So $\mathbf{d}^\bullet\precapprox\mathbf{I}\ \Leftrightarrow\ \lim\limits_{r\rightarrow0}\mathbf{d}^\bullet(r)=0\ \Leftrightarrow\ \forall\,\mathbin{\preceq}\in\Phi^\mathbf{d}\ \exists\,\mathbin{\precsim}\in\Phi^\mathbf{d}\ \forall x\in X\ \exists y\leq^\mathbf{d}\hspace{-3pt}(x\precsim)\ x\preceq y$.

\begin{thm}\label{Cauchytodirected}
If $\mathbf{d}^\bullet\precapprox\mathbf{I}$ then we have $\leq^\mathbf{d}$-directed $Y\subseteq X$ with
\[Y\mathbf{d}=(x_\lambda)\mathbf{d}\qquad\text{and}\qquad\mathbf{d}Y=\mathbf{d}(x_\lambda).\]
\end{thm}

\begin{proof}
As $\mathbf{d}^\bullet\precapprox\mathbf{I}$, i.e. $\lim_{r\rightarrow0}\mathbf{d}^\bullet(r)=0$, we can define $r_n\downarrow0$ with $\mathbf{d}^\bullet(2r_{n+1})<r_n$.

As $(x_\lambda)$ is $\mathbf{d}$-Cauchy, we can define $f:[\Lambda]^{<\omega}\rightarrow\Lambda$ as follows.  Let $f(\{\lambda\})=\lambda$ and, given $F\in[\Lambda]^{<\omega}$ with $|F|>1$, take $f(F)\succ f(E)$, for all $E\subsetneqq F$, such that
\[\sup_{f(F)\prec\lambda}\mathbf{d}(x_{f(F)},x_\lambda)<r_{|F|}.\]
As $\mathbf{d}^\bullet(2r_{|F|})<r_{|F|-1}$, we can take $y_F\leq^\mathbf{d}(x_{f(F)})^\bullet_{2r_{|F|}}$ with $\mathbf{d}(x_{f(F)},y_F)<r_{|F|-1}$.  If $F\subsetneqq G$ then $\mathbf{d}(x_{f(F)},y_G)\leq\mathbf{d}(x_{f(F)},x_{f(G)})+\mathbf{d}(x_{f(G)},y_G)<2r_{|F|}$ and hence $y_G\in(x_{f(F)})^\bullet_{2r_{|F|}}$ so $y_F\leq^\mathbf{d}y_G$.  Thus $Y=\{y_F:F\in[\Lambda]^{<\omega}\}$ is $\leq^\mathbf{d}$-directed.  For $\lambda\succ f(F)$, $x_\lambda\in(x_{f(F)})^\bullet_{r_{|F|}}\subseteq(x_{f(F)})^\bullet_{2r_{|F|}}$ so $y_F\leq^\mathbf{d}x_\lambda$.  Thus $Y\leq^\mathbf{d}(x_\lambda)$ so
\[Y\mathbf{d}\leq(x_\lambda)\mathbf{d}\qquad\text{and}\qquad\mathbf{d}Y\geq\mathbf{d}(x_\lambda).\]
Also $\mathbf{d}(x_{f(F)},y_F)\leq r_{|F|-1}\rightarrow0$ so
\[Y\mathbf{d}\geq(x_\lambda)\mathbf{d}\qquad\text{and}\qquad\mathbf{d}Y\leq\mathbf{d}(x_\lambda).\qedhere\]
\end{proof}

Thus $\mathbf{d}$-completeness follows from $\leq^\mathbf{d}$-$\mathbf{d}$-completeness when $\mathbf{d}^\bullet\precapprox\mathbf{I}$.  As noted after \autoref{edcomplete}, consideration of metric $\mathbf{d}$ shows we can not drop the condition $\mathbf{d}^\bullet\precapprox\mathbf{I}$.  But it does suggest we might replace $\mathbf{d}^\bullet\precapprox\mathbf{I}$ with metric completeness.  More precisely, letting $\mathbf{d}^\mathrm{op}(x,y)=(y,x)$ and $\mathbf{d}^\vee=\mathbf{d}\vee\mathbf{d}^\mathrm{op}$, we might ask if
\begin{equation}\label{q1}
\text{$\mathbf{d}$-complete}\qquad\Leftrightarrow\qquad\text{$\leq^\mathbf{d}$-$\mathbf{d}$-complete}\quad\text{and}\quad\text{$\mathbf{d}^\vee$-complete}?
\end{equation}
In general, the answer is no, as the following simple example shows.

Consider the sequence $(f_m)$ in $[0,\infty]^\mathbb{N}$ defined by
\[f_m(n)=\begin{cases}\infty &\text{if }n<m,\\ 0 &\text{if }n=m\\ 1/n &\text{if }n>m,\end{cases}\]
Set $X=\{f_m:m\in\mathbb{N}\}$ and $\mathbf{d}(f,g)=\sup(f(n)-g(n))_+$.  Then $\leq^\mathbf{d}$ and $\mathbf{d}^\vee$ become identified with $=$ so $X$ is trivially $\leq^\mathbf{d}$-$\mathbf{d}$-complete and $\mathbf{d}^\vee$-complete, even though $(f_m)$ is $\mathbf{d}$-Cauchy with no $X^\circ_\circ$-limit in $X$ ($f_m\carrowc f_\infty$ in $[0,\infty]^\mathbb{N}$ but $f_\infty\notin X$).

Thus if we are to have any hope of proving \eqref{q1}, we need some extra condition.  We could use $\mathbf{d}_\bullet\leq\mathbf{I}$ as in \autoref{Xcompdirected} or the significantly weaker assumption `every open lower ball is directed'.  Again, we can even describe slightly weaker conditions that suffice if we consider the following functions on $[0,\infty]$.
\begin{align*}
\mathbf{d}_\mathbf{F}(r) &=\sup_{x\in X}\sup_{F\in[x_\bullet^r]^{<\omega}}\inf_{F\leq^\mathbf{d}y}\mathbf{d}(y,x).\\
\mathbf{d}_\Phi(r) &=\sup_{x\in X}\sup_{F\in[x_\bullet^r]^{<\omega}}\sup_{\mathbin{\preceq}\in\Phi^\mathbf{d}}\inf_{F\preceq y}\mathbf{d}(y,x).
\end{align*}
\begin{align}
\text{So}\quad\mathbf{d}_\mathbf{F}\leq\mathbf{I}\quad&\Leftrightarrow\quad\mathbf{d}_\mathbf{F}[0,r)\subseteq[0,r),\text{ for all }r\in(0,\infty).\nonumber\\
\quad&\Leftrightarrow\quad x_\bullet^r\text{ is $\leq^\mathbf{d}$-directed, for all }x\in X\text{ and }r\in(0,\infty).\label{xbr}\\
\quad&\Leftrightarrow\quad\sup_{y\in F}\mathbf{d}(y,x)=\inf_{F\leq^\mathbf{d}y}\mathbf{d}(y,x),\text{ for all }x\in X\text{ and finite }F\subseteq X.\nonumber
\end{align}
In general, $\mathbf{d}_\Phi\leq\mathbf{d}_\mathbf{F}\leq\mathbf{d}_\bullet$, but $\mathbf{d}_\mathbf{F}$ can be much smaller than $\mathbf{d}_\bullet$.  For example, if $X=\mathrm{c}_0(\mathbb{R})=\{f\in\mathbb{R}^\mathbb{N}:f(n)\rightarrow0\}$ and $\mathbf{d}(f,g)=\sup(f(n)-g(n))_+$ then $\mathbf{d}_\mathbf{F}\leq\mathbf{I}$ even though $\mathbf{d}_\bullet(r)=\infty$, for all $r>0$.  However, $\mathbf{d}_\mathbf{F}$ and $\mathbf{d}_\Phi$ often coincide.

\begin{prp}\label{dFdF}
If $\mathbf{d}$ is a hemimetric, $X$ is $\mathbf{d}^\vee\!$-complete and $\mathbf{d}_\Phi\precapprox\mathbf{I}$ then $\mathbf{d}_\mathbf{F}=\mathbf{d}_\Phi$.
\end{prp}

\begin{proof}
For any $r\in[0,\infty]$, $x\in X$, finite $F\subseteq x^r_\bullet$ and $\epsilon>0$, we take $(\epsilon_n)$ with $0<\epsilon_n<2^{-n}\epsilon$ and $\mathbf{d}_\Phi(\epsilon_n)<2^{-n}\epsilon$, for all $n\in\mathbb{N}$.  Now take $x_1\in X$ with $\mathbf{d}(x_1,x)<\mathbf{d}_\Phi(r)$ and $\sup_{z\in F\cup\{x\}}\mathbf{d}(z,x_1)<\epsilon_1$.  We can then take $x_2\in X$ with $\mathbf{d}(x_2,x_1)<\mathbf{d}_\Phi(\epsilon_1)<\frac{1}{2}\epsilon$ and $\sup_{z\in F\cup\{x,x_1\}}\mathbf{d}(z,x_2)<\epsilon_2$.  Continuing in this way we obtain $(x_n)$ with $\mathbf{d}^\vee(x_{n+1},x_n)<2^{-n}\epsilon$, for all $n\in\mathbb{N}$.  As $X$ is $\mathbf{d}^\vee$-complete, we have $y\in X$ with $\mathbf{d}^\vee(x_n,y)\rightarrow0$ so $F\leq^\mathbf{d}y$ and $\mathbf{d}(y,x)<\mathbf{d}_\Phi(r)+\epsilon$ so $\mathbf{d}_\mathbf{F}\leq\mathbf{d}_\Phi$.
\end{proof}

\begin{thm}\label{CDS}
If $X$ is $\leq^\mathbf{d}$-$\mathbf{d}$-complete and $\mathbf{d}_\mathbf{F}\leq\mathbf{I}$ then we have $\mathbf{d}^\vee\!$-Cauchy $(y_n)$ with
\[(x_\lambda)\mathbf{d}=(y_n)\mathbf{d}\qquad\text{and}\qquad\lim_{\lambda,n}\mathbf{d}(x_\lambda,y_n)=0.\]
\end{thm}

\begin{proof}
Instead of $\mathbf{d}_\mathbf{F}\leq\mathbf{I}$, we can work with a slightly even weaker condition
\begin{equation}\label{weakcond}
0\in\overline{\{r\in(0,\infty):\mathbf{d}_\mathbf{F}[0,r)\subseteq[0,r)\}},
\end{equation}
which means we have $r_n\downarrow0$ with $\mathbf{d}_\mathbf{F}[0,r_n)\subseteq[0,r_n)$, for all $n\in\mathbb{N}$.  Then we have $(r^m_n)$ with $\mathbf{d}_\mathbf{F}(r^m_n)<r^{m+1}_n<r_n$, for all $m\in\mathbb{N}$ (taking $\mathbf{d}_\mathbf{F}(r_n^0)=0$).  Set
\[\epsilon_n^m=\tfrac{1}{2}(r_n^m-\mathbf{d}_\mathbf{F}(r_n^{m-1})).\]

Again define a map $f:[\Lambda]^{<\omega}\rightarrow\Lambda$ such that, for all $\lambda\in\Lambda$, $f(\{\lambda\})=\lambda$, for all $F\in[\Lambda]^{<\omega}$ with $|F|>1$ and all $E\subsetneqq F$, $f(E)\prec f(F)$ and
\[\sup_{f(F)\prec\lambda}\mathbf{d}(x_{f(F)},x_\lambda)<\min_{1\leq n<|F|}\epsilon_n^{|F|-n},\]

For any $n\in\mathbb{N}$, let $\Lambda_n=\{F\in[\Lambda]^{<\omega}:|F|>n\}$ and define $(y^n_F)_{F\in\Lambda_n}$ recursively as follows.  When $|F|=n+1$, let $y^n_F=x_{f(F)}$ so if $F\subsetneqq G$ then
\[\mathbf{d}(y^n_F,x_{f(G)})<\epsilon_n^1<r_n^1.\]
When $|G|=n+2$, we take $y^n_G$ with $y^n_F\leq^\mathbf{d}y^n_G$, for all $F\subsetneqq G$ with $|F|=n+1$, and
\[\mathbf{d}(y^n_G,x_{f(G)})<\mathbf{d}_\mathbf{F}(r_n^1)+\epsilon^2_n.\]
As $\mathbf{d}(x_{f(G)},x_{f(H)})<\epsilon^2_n$, whenever $G\subsetneqq H$ and $|G|=n+2$,
\[\mathbf{d}(y^n_G,x_{f(H)})\leq\mathbf{d}(y^n_G,x_{f(G)})+\mathbf{d}(x_{f(G)},x_{f(H)})<\mathbf{d}_\mathbf{F}(r_n^1)+2\epsilon^2_n=r_n^2.\]
For $|H|=n+3$, take $y^n_H$ with $y^n_G,x_{f(G)}\leq^\mathbf{d}y^n_H$, for $G\subsetneqq H$ with $|G|=n+2$, and
\[\mathbf{d}(y^n_H,x_{f(H)})<\mathbf{d}_\mathbf{F}(r_n^2)+\epsilon^3_n.\]

Continuing in this way we obtain increasing $(y^n_F)$ with $\mathbf{d}(y^n_F,x_{f(F)})<r_n$ and $x_{f(F)}\leq^\mathbf{d}y^n_G$, for all $F\in\Lambda_{n+1}$ and $F\subsetneqq G$.  As $X$ is $\leq^\mathbf{d}$-$\mathbf{d}$-complete, $(y^n_F)$ has $\mathbf{d}$-supremum $y^n$.  For all $m,n\in\mathbb{N}$ and $F\in\Lambda_{\max(m,n)+1}$, we have
\[\mathbf{d}(y^m_F,y^n)\leq\mathbf{d}(y^m_F,y^n_F)\leq\mathbf{d}(y^m_F,x_{f(F)})<r_m\]
and hence $\mathbf{d}(y^m,y^n)\leq r_m$, so $(y^n)$ is $\mathbf{d}^\vee$-Cauchy.
For any $\epsilon>0$, we have $r_n<\epsilon$ for all sufficiently large $n\in\mathbb{N}$.  Then, for any $z\in X$ and all sufficiently large $F\in[\Lambda]^{<\omega}$,
\[\mathbf{d}(y^n,z)\leq\mathbf{d}(y^n_F,z)+\epsilon\leq\mathbf{d}(x_{f(F)},z)+r_n+\epsilon<\mathbf{d}(x_{f(F)},z)+2\epsilon,\]
so $(y^n)\mathbf{d}\leq(x_\lambda)\mathbf{d}$.  For all sufficiently large $F\in[\Lambda]^{<\omega}$, $\sup_{f(F)\prec\lambda}\mathbf{d}(x_{f(F)},x_\lambda)<\epsilon$ so, as $x_{f(G)}\leq^\mathbf{d}y^n$ when $F\subsetneqq G\in\Lambda_n$,
$\mathbf{d}(x_{f(F)},y^n)<\epsilon$ and $\lim\limits_{\lambda,n}\mathbf{d}(x_\lambda,y^n)=0$.
\end{proof}

Above we obtained symmetric $\mathbf{d}^\vee$ and transitive $\leq^\mathbf{d}$ from $\mathbf{d}$.  But in practice it often happens the other way around, i.e. we compose symmetric $\mathbf{e}$ with transitive $\preceq$ to obtain $\mathbf{d}=\mathbf{e}\circ\mathbin{\preceq}$ (\eqref{tri} is not automatic but follows from e.g. $\mathbf{e}\circ\mathbin{\preceq}=\mathbin{\preceq}\circ\mathbf{e}$).

\begin{qst}\label{mainq}
If $\mathbf{d}=\mathbf{e}\circ\mathbin{\leq^\mathbf{d}}$ for a metric $\mathbf{e}$ then does
\[\text{$\leq^\mathbf{d}$-$\mathbf{d}$-complete}\quad\text{and}\quad\text{$\mathbf{e}$-complete}\qquad\Rightarrow\qquad\text{$\mathbf{d}$-complete}?\]
\end{qst}

Unlike with \eqref{q1}, we do not know of a counterexample.  Indeed, an answer to \autoref{mainq} would likely shed some light on an old problem from \cite{AkemannPedersen1973} and \cite{Brown1988} for C*-algebra $A$, namely whether every strongly lower semicontinuous element of $A^{**}_\mathrm{sa}$ can be obtained from $A_\mathrm{sa}$ as a monotone limit.  However, we can give a positive answer to \autoref{mainq} if we assume $\mathbf{e}$-separability, i.e. $\mathbf{e}Y=0$ for some countable $Y\subseteq X$, or consider $\mathbf{d}$-$\mathbf{d}$-completeness instead of $\leq^\mathbf{d}$-$\mathbf{d}$-completeness.

Again we work with a weaker assumption than $\mathbf{d}=\mathbf{e}\circ\mathbin{\leq}^\mathbf{d}$ which depends only on $\Phi^\mathbf{d}$ and $\Phi^\mathbf{e}$.  Specifically, note $\mathbf{d}\precapprox\mathbf{e}\,\Leftrightarrow\,\Phi^\mathbf{d}\subseteq\Phi^\mathbf{e}$ and define
\[\mathbf{e}\circ\Phi^\mathbf{d}=\sup\limits_{\preceq\in\Phi^\mathbf{d}}\mathbf{e}\circ\mathbin{\preceq}=\sup\limits_{\preceq\in\Phi^\mathbf{d}}\inf\limits_{z\preceq y}\mathbf{e}(x,z)=\sup\limits_{\epsilon>0}\inf\limits_{z<^\mathbf{d}_\epsilon y}\mathbf{e}(x,z).\]

\begin{thm}\label{ded}
If $X$ is $\mathbf{e}$-complete and $\mathbf{e}\circ\Phi^\mathbf{d}\precapprox\mathbf{d}\precapprox\mathbf{e}=\mathbf{e}^\mathrm{op}$ then $\mathbf{e}\circ\Phi^\mathbf{d}=\mathbf{e}\circ\mathbin{\leq^\mathbf{d}}$,
\begin{equation}\label{dedeq}
Y\mathbf{d}=(x_\lambda)\mathbf{d}\qquad\text{and}\qquad\mathbf{d}Y=\mathbf{d}(x_\lambda),
\end{equation}
for $\mathbf{d}$-directed $Y\subseteq X$.  If $X$ is $\mathbf{e}$-separable then we can choose $Y$ to be $\leq^\mathbf{d}$-directed.
\end{thm}

\begin{proof}
For $\mathbf{e}\circ\Phi^\mathbf{d}=\mathbf{e}\circ\mathbin{\leq^\mathbf{d}}$, we argue as in the proof of \autoref{dFdF}. Specifically, for any $x,y\in X$ and $\epsilon>0$, take $\epsilon_n\downarrow0$ with $^{\mathbf{e}\circ\Phi^\mathbf{d}}\!/\!_\mathbf{d}(\epsilon_n)<2^{-n}\epsilon$, for all $n\in\mathbb{N}$.  Now take $z_1\in X$ with $\mathbf{e}(x,z_1)<\mathbf{e}\circ\Phi^\mathbf{d}(x,y)+\epsilon$ and $\mathbf{d}(z_1,y)<\epsilon_1$.  Thus $\mathbf{e}\circ\Phi^\mathbf{d}(z_1,y)<\frac{1}{2}\epsilon$ and we can take $z_2\in X$ such that $\mathbf{e}(z_1,z_2)<\frac{1}{2}\epsilon$ and $\mathbf{d}(z_2,y)<\epsilon_2$.  Continuing in this way we obtain a sequence $(z_n)\subseteq X$ such that, for all $n\in\mathbb{N}$,
\[\mathbf{e}(z_n,z_{n+1})\leq2^{-n}\epsilon\qquad\text{and}\qquad\mathbf{d}(z_n,y)<\epsilon_n\rightarrow0.\]
As $X$ is $\mathbf{e}$-complete, $\mathbf{e}(z_n,z)\rightarrow0$, for some $z\in X$, so
\[\mathbf{e}(x,z)\leq\mathbf{e}(x,z_1)+\mathbf{e}(z_1,z)\leq\mathbf{e}\circ\Phi^\mathbf{d}(x,y)+2\epsilon\]
As $\mathbf{e}=\mathbf{e}^\mathrm{op}$, $\mathbf{e}(z,z_n)\rightarrow0$ so, as $\mathbf{d}\precapprox\mathbf{e}$, $\mathbf{d}(z,z_n)\rightarrow0$.  Then $z\leq^\mathbf{d}y$ follows from $\mathbf{d}(z,y)\leq\mathbf{d}(z,z_n)+\mathbf{d}(z_n,y)\rightarrow0$.  As $\epsilon>0$ was aribtrary, $\mathbf{e}\circ\mathbin{\leq^\mathbf{d}}=\mathbf{e}\circ\Phi^\mathbf{d}$.

As $(x_\lambda)$ is Cauchy, we can take a subnet and $(s_\lambda),(t_\lambda)\subseteq(0,\infty)$ such that
\begin{align*}
\sup_{\lambda<\delta}\mathbf{d}(x_\lambda,x_\delta)&<s_\lambda\rightarrow0.\\
^{\mathbf{e}\circ\mathbin{\leq^\mathbf{d}}}\!/\!_\mathbf{d}(s_\lambda)&<t_\lambda\rightarrow0.
\end{align*}
Define $\gamma_\lambda^n$ and $x_\lambda^n\leq^\mathbf{d}x_{\gamma_\lambda^n}$ recursively as follows.  First set $\gamma_\lambda^1=\lambda$ and $x_\lambda^1=x_\lambda$.  Then, for all $n\in\mathbb{N}$, take $\gamma_\lambda^{n+1}>\gamma_\lambda^n$ such that $^{\mathbf{e}\circ\mathbin{\leq^\mathbf{d}}}\!/\!_\mathbf{d}(s_{\gamma_\lambda^{n+1}}),s_{\gamma_\lambda^{n+1}}<2^{-n}t_\lambda$.  As $\mathbf{d}(x_\lambda^n,x_{\gamma_\lambda^{n+1}})\leq\mathbf{d}(x_{\gamma_\lambda^n},x_{\gamma_\lambda^{n+1}})<s_{\gamma_\lambda^n}$ and $^{\mathbf{e}\circ\mathbin{\leq^\mathbf{d}}}\!/\!_\mathbf{d}(s_{\gamma_\lambda^n})<2^{1-n}t_\lambda$, we can take $x^{n+1}_\lambda\leq^\mathbf{d}x_{\gamma_\lambda^{n+1}}$ such that $\mathbf{e}(x^n_\lambda,x^{n+1}_\lambda)<2^{1-n}t_\lambda$.  For each $\lambda$, $(x_\lambda^n)$ is $\mathbf{e}$-Cauchy so $\mathbf{e}$-completeness implies that $\mathbf{e}(x_\lambda^n,y_\lambda)\rightarrow0$, for some $y_\lambda\in X$.

For any $\lambda$ and $\epsilon>0$, we can take $n$ with $2^{1-n}t_\lambda<\epsilon$ so $\mathbf{e}(x_\lambda^n,y_\lambda)<2\epsilon$ and $\mathbf{d}(x_{\gamma_\lambda^n},x_\delta)<s_{\gamma_\lambda^n}<\epsilon$, for any $\delta\succ\gamma_\lambda^n$.  For all sufficiently large $\delta$, we also have $t_\delta<\epsilon$ so $\mathbf{e}(x_\delta,y_\delta)<2\epsilon$ and hence
\begin{align*}
\mathbf{d}(y_\lambda,y_\delta)&\leq\mathbf{d}(y_\lambda,x_\lambda^n)+\mathbf{d}(x_\lambda^n,x_{\gamma_\lambda^n})+\mathbf{d}(x_{\gamma_\lambda^n},x_\delta)+\mathbf{d}(x_\delta,y_\delta)\\
&\leq{}^\mathbf{d}\!/\!_\mathbf{e}(2\epsilon)+0+\epsilon+{}^\mathbf{d}\!/\!_\mathbf{e}(2\epsilon).
\end{align*}
As $\mathbf{d}\precapprox\mathbf{e}$, $Y=\{y_\lambda:\lambda\in\Lambda\}$ is $\mathbf{d}$-directed.  As $\mathbf{e}(x_\lambda,y_\lambda)<2t_\lambda\rightarrow0$, \eqref{dedeq} follows.

If $X$ is $\mathbf{e}$-separable then $\mathbf{e}$ is a pseudometric, as $\mathbf{e}=\mathbf{e}^\mathrm{op}$.  Thus $Y$ is also $\mathbf{e}$-separable and can be replaced by a countable subset.  Then we can replace $(x_\lambda)$ with a $\mathbf{d}$-Cauchy sequence $(x_n)\subseteq Y$ with $Y\leq^\mathbf{d}(x_n)$.

Take $(s^m_n),(t^m_n)\subseteq(0,\infty)$ such that, for all $m,n\in\mathbb{N}$,
\[s^m_n<2^{-m-n},\quad{}^\mathbf{d}\!/\!_\mathbf{e}(s^m_n)<t^m_{n-1}\quad\text{and}\quad ^{\mathbf{e}\circ\mathbin{\leq^\mathbf{d}}}\!/\!_\mathbf{d}(t^m_n)<s^{m+1}_n\]
(define and $(s^m_1)_{m\in\mathbb{N}}$ first then $(t^m_1)_{m\in\mathbb{N}}$, $(s^m_2)_{m\in\mathbb{N}}$ etc.).  Take a subsequence $(x_n)$ with $\mathbf{d}(x_n,x_{n+1})<t^1_n$, for all $n$, and define $y^m_n$ with $\mathbf{d}(y^m_n,y^m_{n+1})<t^m_n$, for all $m$, recursively as follows.  First let $y^1_n=x_n$, for all $n$.  Assume $y^m_n$ is defined for all $n$ and fixed $m$.  For each $n$, we can take $y^{m+1}_n\leq^\mathbf{d}y^m_{n+1}$ with $\mathbf{e}(y^m_n,y^{m+1}_n)<s^{m+1}_n$ as
\[\mathbf{e}\circ\mathbin{\leq^\mathbf{d}}(y^m_n,y^m_{n+1})\leq{}^{\mathbf{e}\circ\mathbin{\leq^\mathbf{d}}}\!/\!_\mathbf{d}(\mathbf{d}(y^m_n,y^m_{n+1}))\leq{}^{\mathbf{e}\circ\mathbin{\leq^\mathbf{d}}}\!/\!_\mathbf{d}(t^m_n)<s^{m+1}_n.\]
Thus $\mathbf{d}(y^{m+1}_n,y^{m+1}_{n+1})\leq\mathbf{d}(y^m_{n+1},y^{m+1}_{n+1})\leq{}^\mathbf{d}\!/\!_\mathbf{e}(\mathbf{e}(y^m_{n+1},y^{m+1}_{n+1}))\leq{}^\mathbf{d}\!/\!_\mathbf{e}(s^{m+1}_{n+1})<t^{m+1}_n$.

For all $m,n\in\mathbb{N}$, $\mathbf{e}(y^m_n,y^{m+1}_n)<s_n^{m+1}<2^{-m-n}$ so, as $X$ is $\mathbf{e}$-complete, we have $y_n\in X$ with $\lim_m\mathbf{e}(y^m_n,y_n)=0$.  As $\mathbf{d}\precapprox\mathbf{e}=\mathbf{e}^\mathrm{op}$ and $y^{m+1}_n\leq^\mathbf{d}y^m_{n+1}$,
\[\mathbf{d}(y_n,y_{n+1})\leq\liminf_m(\mathbf{d}(y_n,y^{m+1}_n)+\mathbf{d}(y^{m+1}_n,y^m_{n+1})+\mathbf{d}(y^m_{n+1},y_{n+1}))=0,\]
i.e. $y_n\leq^\mathbf{d}y_{n+1}$ so $Y=\{y_n:n\in\mathbb{N}\}$ is $\leq^\mathbf{d}$-directed.  Lastly, \eqref{dedeq} follows from
\[\mathbf{e}(x_n,y_n)=\lim_m\mathbf{e}(x_n,y^m_n)<{\textstyle\sum\limits_{m=2}^\infty}s^m_n<{\textstyle\sum\limits_{m=2}^\infty}2^{-m-n}<2^{-n}\rightarrow0.\qedhere\]
\end{proof}

\begin{cor}\label{Yc} $X$ is $\mathbf{d}$-complete if any of the following hold.
\begin{enumerate}
\item\label{Yc1a} $X$ is $\leq^\mathbf{d}$-$\mathbf{d}$-complete and $\mathbf{d}^\bullet\,\precapprox\mathbf{I}$.
\item\label{Yc1b} $X$ is $\leq^\mathbf{d}$-$\mathbf{d}$-complete, $\mathbf{d}^\vee\!$-complete and $\mathbf{d}_\mathbf{F}\leq\mathbf{I}$.
\item\label{Yc2a} $X$ is\hspace{10pt}$\mathbf{d}$-$\mathbf{d}$-complete, $\mathbf{e}$-complete and $\mathbf{e}\circ\Phi^\mathbf{d}\precapprox\mathbf{d}\precapprox\mathbf{e}=\mathbf{e}^\mathrm{op}$.
\item\label{Yc2b} $X$ is $\leq^\mathbf{d}$-$\mathbf{d}$-complete, $\mathbf{e}$-complete, $\mathbf{e}$-separable and $\mathbf{e}\circ\Phi^\mathbf{d}\precapprox\mathbf{d}\precapprox\mathbf{e}=\mathbf{e}^\mathrm{op}$.
\end{enumerate}
\end{cor}

\begin{proof}  If $\mathbf{d}^\bullet\,\precapprox\mathbf{I}$ then, for any $\mathbf{d}$-Cauchy $(x_\lambda)$, we have $Y\subseteq X$ with $Y\mathbf{d}=(x_\lambda)\mathbf{d}$, by \autoref{Cauchytodirected}.  If $X$ is also $\leq^\mathbf{d}$-$\mathbf{d}$-complete then we have $x=\mathbf{d}$-$\sup Y$ and hence $x\mathbf{d}=Y\mathbf{d}=(x_\lambda)\mathbf{d}$ so $x_\lambda\carrowb x\leq^\mathbf{d}x$, i.e. $x_\lambda\carrowc x$, by \eqref{carrowc}.  This proves \eqref{Yc1a} and likewise \eqref{Yc1b} follows from \autoref{CDS}, while \eqref{Yc2a} and \eqref{Yc2b} follow from \autoref{ded}.
\end{proof}

Note in \autoref{Yc} \eqref{Yc1b}, if $\mathbf{d}$ is a hemimetric then we can replace $\mathbf{d}_\mathbf{F}$ with $\mathbf{d}_\Phi$ for a formally weaker assumption (even weaker if we consider \eqref{weakcond}), by \autoref{dFdF}.

For a simple application of \autoref{Yc} \eqref{Yc1a}, we consider the space of `generalized formal balls' of $X$.  Specifically, identify $X$ with $X\times\{0\}$ and extend $\mathbf{d}$ to $X\times\mathbb{R}$ by
\[\mathbf{d}((x,r),(y,s))=(\mathbf{d}(x,y)+r-s)_+.\]
For any $x,y\in X$, $r,s\in\mathbb{R}$ and $t\in[0,\infty)$,
\begin{align*}
\mathbf{d}(x,y)+r-s\leq t\quad&\Leftrightarrow\quad\mathbf{d}((x,r),(y,s))\leq t.\\
\Leftrightarrow\quad\mathbf{d}(x,y)+r-t-s\leq0\quad&\Leftrightarrow\quad(x,r-t)\leq^\mathbf{d}(y,s).\\
\Leftrightarrow\quad\mathbf{d}(x,y)+r-(t+s)\leq0\quad&\Leftrightarrow\quad(x,r)\leq^\mathbf{d}(y,t+s).
\end{align*}
So finite radius closed upper balls have minimums and likewise for lower balls, i.e.
\[\overline{(x,r)}^\bullet_t\ =\ (x,r-t)\!\leq^\mathbf{d}\qquad\text{and}\qquad\leq^\mathbf{d}\!(y,t+s)\ =\ \overline{(y,s)}_\bullet^t.\]
Thus $\mathbf{d}^\bullet\leq\mathbf{I}$ and $\mathbf{d}_\bullet\leq\mathbf{I}$.  And $\mathbf{d}^\bullet\leq\mathbf{I}$ still applies to $X\times\mathbb{R}_-$, where $\mathbb{R}_-=(-\infty,0]$.

\begin{thm}[\cite{KostanekWaszkiewicz2011} Theorem 7.1]\label{KW}
The following are equivalent.
\begin{enumerate}
\item\label{KW1} $X$ is $\mathbf{d}$-complete.
\item\label{KW2} $X\times\mathbb{R}_-$ is $\mathbf{d}$-complete.
\item\label{KW3} $X\times\mathbb{R}_-$ is $\leq^\mathbf{d}$-complete.
\end{enumerate}
\end{thm}

\begin{proof}\
\begin{itemize}
\item[\eqref{KW1}$\Rightarrow$\eqref{KW2}] If $(x_\lambda,r_\lambda)$ is $\mathbf{d}$-Cauchy then, as $(r-s)_+\leq\mathbf{d}((x,r),(y,s))$ and $\mathbb{R}_-$ is bounded above by $0$, $(r_\lambda)$ must be Cauchy (for the usual metric on $\mathbb{R}$).  Thus $r_\lambda\rightarrow r$ for some $r\in\mathbb{R}_-$, and hence $(x_\lambda)$ is $\mathbf{d}$-Cauchy.  Thus $x_\lambda\carrowc x$, for some $x\in X$, and hence $(x_\lambda,r_\lambda)\carrowc(x,r)$ in $X\times\mathbb{R}_-$.
\item[\eqref{KW2}$\Rightarrow$\eqref{KW1}] Identify $X$ with $X\times\{0\}$.
\item[\eqref{KW2}$\Rightarrow$\eqref{KW3}] Immediate.
\item[\eqref{KW3}$\Rightarrow$\eqref{KW2}]  We claim that any $\leq^\mathbf{d}$-supremum $(x,r)$ of $\leq^\mathbf{d}$-directed $(x_\lambda,r_\lambda)$ in $X\times\mathbb{R}_-$ remains a $\leq^\mathbf{d}$-supremum in $X\times\mathbb{R}$.  Indeed, say $(x_\lambda,r_\lambda)\leq^\mathbf{d}(y,s)\in X\times\mathbb{R}$, for all $\lambda$.  As $X\times\mathbb{R}_-$ is $\leq^\mathbf{d}$-complete, we have $(z,t)=\mathbin{\leq^\mathbf{d}}$-$\sup(x_\lambda,r_\lambda-s)$ in $X\times\mathbb{R}_-$, so $(z,t+s)=\mathbin{\leq^\mathbf{d}}$-$\sup(x_\lambda,r_\lambda)=(x,r)$ and hence $(x,r-s)=\mathbin{\leq^\mathbf{d}}$-$\sup(x_\lambda,r_\lambda-s)$.  Also $(x_\lambda,r_\lambda-s)\leq^\mathbf{d}(y,0)$, for all $\lambda$, so $(x,r-s)\leq^\mathbf{d}(y,0)$ and hence $(x,r)\leq^\mathbf{d}(y,s)$, proving the claim.  Thus $(x,r)=\mathbf{d}$-$\sup(x_\lambda,r_\lambda)$ in $X\times\mathbb{R}$, by \autoref{Xcompdirected}, and hence in $X\times\mathbb{R}_-$.  This shows that $X\times\mathbb{R}_-$ is $\leq^\mathbf{d}$-$\mathbf{d}$-complete and hence $\mathbf{d}$-complete, by \autoref{Yc} \eqref{Yc1a}.\qedhere
\end{itemize}
\end{proof}

\bibliography{maths}{}
\bibliographystyle{alphaurl}

\end{document}